\documentclass[11pt,A4]{article}
\usepackage{amsmath,amsthm,amsfonts,amssymb,amscd, amsxtra,mathrsfs, color}
\usepackage[english]{babel}

\usepackage[margin=1.0in]{geometry}
\newtheorem{theorem}{Theorem}
\newtheorem{lemma}[theorem]{Lemma}

\newtheorem{proposition}[theorem]{Proposition}
\newtheorem{remark}[theorem]{Remark}
\newtheorem{definition}[theorem]{Definition}
\newtheorem{example}[theorem]{Example}

\newtheorem{algorithm}{Algorithm}

\begin{document}

\title{Newton method for finding a singularity of a special class of locally Lipschitz continuous vector fields on Riemannian manifolds}

\author{Fabiana R. de Oliveira \thanks{Universidade Federal de Goi\'as, Goi\^ania, GO 74001-970,  BR (e-mail:{\tt fabianardo@gmail.com}). The author was supported in part by CAPES.}
\and
Orizon P. Ferreira,\thanks{ Universidade Federal de Goi\'as, Goi\^ania, GO 74001-970, BR (e-mail:{\tt orizon@ufg.br}). The author was supported in part by CNPq Grants 305158/2014-7, 408151/2016-1 and FAPEG/GO.}
}
\maketitle
\begin{abstract}
In this paper, we extend some results of nonsmooth analysis from  Euclidean context to the Riemannian setting. In particular, we discuss the concept and some properties of  locally Lipschitz continuous vector fields  on  Riemannian settings, such as  Clarke generalized covariant derivative,   upper semicontinuity  and Rademacher theorem. We also present a version of Newton method for finding a singularity of a special class of  locally Lipschitz continuous vector fields. Under mild  conditions, we  establish     the well-definedness and local convergence of the sequence generated by the method in a neighborhood of a singularity.  In particular, a local convergence result for  semismooth vector fields is presented. Furthermore,  under Kantorovich-type assumptions  the convergence of the  sequence generated by the Newton method  to a solution is established, and its  uniqueness  in a suitable neighborhood of the starting point is verified.
\end{abstract}



\section{Introduction}\label{sec:int}
During the last decade, there has been an increasing the number of papers that have proposed extensions of concepts and techniques of nonsmooth analysis,  as well as  nonsmooth methods, from Euclidean context to the Riemannian setting, these paper  include   but not limited to \cite{Azagra2005, BacakEtall2016,  HosseiniEtall2017, GrohsHossein2016,  Hosseini2013,   Hosseini2017,    LedyaevYuZhu2007}. In these works, some concepts and classical results, such as generalized gradient or Clarke subdifferential, generalized directional derivative, chain rule and Lebourg mean value theorem have been generalized to Riemannian context. In addition, different methods have been proposed for solving nonsmooth optimization problems on Riemannian manifolds. For instance,  in  \cite{Hosseini2017} was  proposed the gradient sampling algorithm for finding a minimizer of nonsmooth locally Lipschitz functions,  in  \cite{GrohsHossein2016} was presented a Riemannian trust region method for unconstrained optimization problems, whose objective functions are locally Lipschitz.  One the main reason of the increasing interest to develop theoretical  and computational tools is the fact that nonsmooth optimization problems arise in a variety of applications, such as in computer vision, robotics, signal processing, see \cite{BacakEtall2016,  BergmannWeinmann2016}.  Although interest in nonsmooth analysis in the Riemannian setting  has increased,  there are few papers on nonsmooth  vector fields in this context,  see \cite{HosseiniEtall2017, Rampazzo2007}. However, as the nonsmooth vector fields can be viewed as a natural generalization of nonsmooth vector-valued maps, we believe that the development of this theory is great interest.

Our purpose in this paper is present a version of Newton method for finding a singularity of  locally Lipschitz continuous vector fields.  In order to present our method, we first define the Clarke generalized covariant derivative, which can be viewed as a natural generalization to Riemannian setting of Clarke generalized Jacobian, studied in \cite{Clarke1990}.  The concept of Clarke generalized covariant derivative has already appeared in  \cite{HosseiniEtall2017, Rampazzo2007}. In the present paper, we show its existence using a version of Rademacher theorem in the Riemannian setting, which  is one of our  contributions.

It is well-known that the Newton method is the method more popular for finding a singularity of a differentiable vector field, the origins of which go back to the work of  \cite{Shub1984}, see also  \cite{Gabay1982, ChongWang2006, Smith1994, Udricste1994, Wang2011}. This method became popular owing to its attractive convergence properties under suitable assumptions. For instance, in all previously cited works  the superlinear and/or quadratic local convergence of the sequence generated by the Newton method have been established under  invertibility of the covariant derivative of the vector field at its singularity and/or  Lipschitz-like conditions on the covariant derivative of the vector field.  Recently, in \cite{FernandesFerreiraYuan2017} were established local properties of  Newton method  under invertibility of the covariant derivative of the vector field at its singularity. Basically, in the Newton method the vector field is replaced by an approximation depending on the current iterate, and then the original problem is converted in an approximated problem which can be solved more easily. The solution of this approximated problem is then taken as a new iterate and the process is repeated. The success of Newton method for finding a singularity of a differentiable vector field motivates us to study Newton method for  finding a singularity of a locally Lipschitz continuous vector field. The essence of the our method is similar to classical case, however, in the approximated problem we combines the exponential mapping on the manifold with an element of Clarke generalized covariant derivative of the vector field. This is because, the derivative covariant of a locally Lipschitz continuous vector field may not exist. It is worth to pointed out that when the vector field is continuously differentiable our method reduces to the classical Newton method. From the theoretical viewpoint, we present local and semi-local convergence analysis of the proposed method under mild   assumptions.

The paper is organized as follows. In Section~\ref{sec:int.1}, some notation and basic results are presented. In Section~\ref{sec:na}, we generalize  some results of nonsmooth analysis to Riemannian context, in particular, we establish the Rademacher theorem, introduce the Clarke generalized covariant derivative associated to a locally Lipschitz continuous vector field  and study  its main properties. In Section~\ref{sec:nmca}, we describe the Newton method and establish its convergence theorems. In Section~\ref{sec:exam}, we present a class of examples of locally Lipschitz continuous vector field satisfying the assumptions of the convergence theorems. We conclude the paper with some remarks in Section~\ref{sec:fr}.
\section{Notation and Auxiliary Results} \label{sec:int.1}
In this section, we recall some notations, definitions and basic properties of Riemannian manifolds used throughout the paper. They can be found in many books on Riemannian Geometry, for example, in \cite{Lang1995, Sakai1996, Loring2011}.

A \textit{chart} on a $n$-dimensional  smooth manifold $M$ is a pair $(U, \varphi)$, where $U$ is an open subset of $M$ and the {\it coordinate map} $\varphi: U \to \widehat{U}$ is a  smooth homeomorphism from $U$ to an open subset $\widehat{U} = \varphi(U) \subseteq \mathbb{R}^n$. Let $N$  and $M$  be  manifolds of finite dimension and  $F: N \to M$  be a  continuous map.   We say that  $F$ is {\it smooth} at $p\in N$, if there exist smooth charts $(U, \varphi)$ containing $p$ and $(W, \psi)$ containing $F(p)$ such that $F(U)\subseteq W$ and the composite  mapping $\psi \circ F \circ \varphi^{-1}: \varphi(U) \to \psi (W)$ is smooth at $\varphi(p)$.  The definition of the  smoothness  of a map $F: N \to M$  at a point is independent of the choice of charts, see \cite[Proposition~6.7]{Loring2011}. A {\it diffeomorphism} of manifolds is a bijective  smooth  map $F: N \to M$ whose inverse $F^{-1}$ is also smooth. According to \cite[Proposition 6.10]{Loring2011} coordinate maps are diffeomorphisms and, in particular,  are continuously differentiable. Let $M$ be a Riemannian manifold with {\it Riemannian metric} denoted by  $\langle  \cdot,   \cdot \rangle$ and the corresponding {\it norm}  by $\|\cdot\|$.  The {\it length} of a piecewise smooth curve  $\gamma:[a,b]\rightarrow M$ joining $p$ to $q$ in $M$, i.e., $\gamma(a)= p$ and $\gamma(b)=q$ is denoted by  $\ell(\gamma)$.  The {\it Riemannian distance} between $p$ and $q$ is defined as
$
d(p,q) = \inf_{\gamma \in \Gamma_{p,q}} \ell(\gamma),
$
where $\Gamma_{p,q}$ is the set of all piecewise smooth curves in $M$ joining points $p$ and $q$. This distance induces the original topology on $M$, namely $(M, d)$ is a complete metric space and the bounded and closed subsets are compact. The {\it open and closed balls} of radius $r>0$ centred at $p$ are defined respectively by $B_{r}(p):=\left\{ q\in M :~ d(p,q)<r\right\}$ and $ B_{r}[p]:=\left\{ q\in M :~d(p,q)\leq r\right\}.$ Consider $M$ a $n$-dimensional smooth Riemannian manifold. Denote  the {\it tangent space} at point $p$  by $T_pM$,   the {\it tangent bundle}  by $TM := \bigcup_{p\in M}T_pM$  and a {\it vector field} by a mapping   $X: M \to TM$ such that $X(p) \in T_pM$.  Let $\gamma$ be a curve joining the points $p$ and $q$ in $M$ and let $\nabla$ be the Levi-Civita connection associated to $(M, \langle \cdot, \cdot \rangle)$. For each $t \in [a,b]$, $\nabla$ induces a linear isometry between the tangent spaces $T _{\gamma(a)} {M}$ and $T _{\gamma(t)} {M}$, relative to $\langle \cdot , \cdot \rangle$, defined by
$P_{\gamma,a,t}v = Y(t)$, where $Y$ is the unique vector field on $\gamma$ such that $\nabla_{\gamma'(t)}Y(t) = 0$ and $Y(a)= v$. This isometry is called {\it parallel transport} along the  segment $\gamma$ joining  $\gamma(a)$ to $\gamma(t)$. It can  be showed   that
$P_{\gamma,\,b_{1},\,b_{2}}\circ P_{\gamma,\,a,\,b_{1}}= P_{\gamma,\,a,\,b_{2}}$ and $P_{\gamma,\,b,\,a}=P^{-1}_{\gamma,\,a,\,b}.$
For simplify and convenience, whenever   there is no confusion we will  consider the notation $P_{\gamma,p,q}$  instead of $P_{\gamma,\,a,\,b}$,  where $\gamma$ is a segment joining $p$ to $q$  with $\gamma(a)=p$ and $\gamma(b)=q$.  We will use the short  notation    $P_{pq}$  instead of  $P_{\gamma,p,q}$ whenever  there exists  an  unique geodesic   segment joining $p$ to $q$. For any $n$-dimensional smooth manifold $M$; the tangent bundle $TM$ has a natural topology and smooth structure that make it into a $2n$-dimensional smooth manifold. With respect to this structure, the projection $\pi: TM \to M$ is smooth, see \cite[Proposition 3.18]{Lee2013}. The {\it standard Riemannian distance $d_{\mbox{\tiny $TM$}}$  on the tangent bundle} $TM$  can  be defined as follows: given $u, v \in TM$, then  $d_{\mbox{\tiny $TM$}}$ is defined by
\begin{equation} \label{eq:MBundle}
d_{\mbox{\tiny $TM$}}(u, v):=\inf \left\{ \sqrt{\ell^2(\gamma)  + \| P_{\gamma, \pi u, \pi v } u - v\|^2} :~ \gamma \in \Gamma_{\pi u,  \pi v} \right\},
\end{equation}
where $\Gamma_{\pi u,  \pi v}$ is the set of all piecewise smooth curves in $M$ joining the points $\pi u$ to $\pi v$, whose derivative is never zero, see \cite[Appendix, p. 240]{CanaryEpsteinMarden2006}. A vector field $Y$ along a smooth curve $\gamma$ in $M$ is said to be {\it parallel}  when $\nabla_{\gamma^{\prime}} Y=0$. If $\gamma^{\prime}$ itself is parallel, we say that $\gamma$ is a {\it geodesic}. The geodesic equation $\nabla_{\gamma'}\gamma' = 0$ is a second-order nonlinear ordinary differential equation, so the geodesic $\gamma$ is determined by its position $p$ and velocity $v$ at $p$. It is easy to check that $\left\| \gamma' \right\|$ is constant. The restriction of a geodesic to a  closed bounded interval is called a {\it geodesic segment}. A geodesic segment  joining $p$ to $q$ in $M$ is said to be {\it minimal} if its length is equal to $d(p,q)$ and, in this case, it will be denoted by $\gamma_{pq}$.   A Riemannian manifold is {\it complete} if its geodesics $\gamma(t)$ are defined for any value of $t\in \mathbb{R}$. The Hopf-Rinow theorem asserts that any pair of points in a complete Riemannian  manifold $M$ can be joined by a (not necessarily unique) minimal geodesic segment. From now on, {\it $M$ denotes  a $n$-dimensional smooth and complete Riemannian manifold}. Due to the completeness of the Riemannian manifold $M$, the {\it exponential map} at $p$, $\exp_{p}:T_{p}{M} \to {M} $ can be given by $\exp_{p}v = \gamma(1)$, where $\gamma$ is the geodesic defined by its position $p$ and velocity $v$ at $p$ and $\gamma(t) = \exp_p(tv)$ for any value of $t$. The inverse of the exponential map (if exists) is denote by $\exp^{-1}_{p}$. Let $p\in {M}$, the {\it injectivity radius} of ${M}$ at $p$ is defined by
$ r_{p}:=\sup\{ r>0:~{\exp_{p}}_{\lvert_{B_{r}(0_{p})}} \mbox{ is\, a\, diffeomorphism} \},$
where $0_{p}$ denotes the origin of $T_{p}{M}$ and $B_{r}(0_{p}):=\lbrace  v\in T_{p}{M}:~\| v-0_{p}\| <r\rbrace$. A neighborhood $\mathcal{W}$ of $p \in M$ is said to be {\it normal neighborhood} of $p$ if there exists   a neighborhood $\mathcal{U}$ of the origin in $T_pM$ such that $\exp_p: \mathcal{U} \to \mathcal{W}$ is a diffeomorphism. Furthermore, if $\mathcal{W}$ is a normal neighborhood of each of its points, then $\mathcal{W}$ is said to be {\it totally normal neighborhood}.
\begin{remark}\label{unicidadedageodesica}
For $\bar{p}\in M$, the above definition implies that if $0<\delta<r_{\bar{p}}$, then $\exp_{\bar{p}}B_{\delta}(0_{ \bar{p}}) = B_{\delta}( \bar{p})$ is a   totally normal neighborhood. Hence, for all $p, q\in B_{\delta}(\bar{p})$, there exists a unique geodesic segment $\gamma$ joining  $p$ to $q$, which is given by $\gamma_{p q}(t)=\exp_{p}(t \exp^{-1}_{p} {q})$, for all $t\in [0, 1]$ and $d(p,q) = \|\exp_p^{-1}q\|$.
\end{remark}
Next we present a quantity, which plays  an important role in the sequel,  it  was defined in \cite{Dedieu2003}.
\begin{definition} \label{def:kp}
 Let $p \in M$ and $r_{p}$ be the radius of injectivity of $M$ at $p$. Define the quantity
$$K_{p}:=\sup\left \{\dfrac{d(\exp_{q}u, \exp_{q}v)}{\left\| u-v\right\|} :~ q\in B_{r_{p}}(p), ~ u,\,v\in T_{q}{M}, ~u\neq v, ~\| v\| \leq r_{p},~
\| u-v\|\leq r_{p}\right\}.$$
\end{definition}
In the next remark, we examine that under suitable  conditions is possible to estimate the value of  $K_p$.
\begin{remark} This number $K_{p}$ measures how fast the geodesics spread apart in $M$. In particular, when $u = 0$ or more generally when $u$ and $v$ are on the same line through $0$, $d(\exp_{q}u, \exp_{q}v)=\| u-v\|$. Hence, $K_{p}\geq1$, for all $p\in M$. When $M$ has non-negative sectional curvature, the geodesics spread apart less than the rays \cite[Chapter 5]{doCarmo1992}, i.e., $d(\exp_{p}u, \exp_{p}v)\leq\|u-v\|$ and, in this case, $K_{p}=1$  for all $p\in M$.
\end{remark}
The {\it directional derivative} of   $X$ at $p$ in the direction $v\in T_{p}M$ is defined by
\begin{align}\label{derdir}
\nabla X(p, v) = \lim_{t\rightarrow 0^{+}}\dfrac{1}{t}\left[P_{\exp_{p}(tv) p}X(\exp_{p}(tv)) - X(p)\right] \in T_pM,
\end{align}
whenever  the limit exists, where $P_{\exp_{p}(tv) p}$  denotes the parallel transport along $\gamma(t)= \exp_{p}(tv)$. If this directional derivative exists for every $v$, then $X$ is said to be \textit{directionally differentiable} at $p$. Denote by $\mathcal{X}(M)$ the space of the  differentiable vector fields on $M$. For each $X\in \mathcal{X}(M)$, the {\it covariant derivative} of $X$ determined by the Levi-Civita connection $\nabla$ defines at each $p\in M$ a linear mapping $\nabla X(p):T_pM\to T_pM$ given by $\nabla X(p)v:=\nabla_{Y}X(p)$, where $Y$ is a vector field such that $Y(p)= v$. Furthermore,    $\nabla X(p, v)= \nabla X(p)v$, see \cite[Proposition 3, p. 234]{Spivak1979}. To state  the next result we need to define the norm of a linear mapping.
\begin{definition}\label{de:normmult}
Let $p \in M$, the  norm of a linear mapping $A: T_pM \to T_pM$  is defined by
$$\|A\|:=\sup \left\{ \|A v  \|:~  v \in T_pM, ~\|v\| = 1 \right\}.$$
\end{definition}
We end this section with   the well-known Banach Lemma. For a prove of it see \cite[Lemma~2.3.2]{Ortega1970}.
\begin{lemma} \label{le:Banach}
Let $A$, $B$ be a linear operators in $T_pM$. If $A$ is non-singular and $\| A^{-1}\|\|B- A\|<1$, then $B$ is non-singular and
$$\|B^{-1}\|\leq  \frac{\|A^{-1}\|}{1 -\|A^{-1}(B-A)\|}.$$
\end{lemma}
\section{Nonsmooth Analysis} \label{sec:na}
The goal of this section is extend some basic results of nonsmooth analysis from linear context  to Riemannian setting.  In particular,  we study the basic  properties of the locally Lipschitz continuous vector fields in Riemannian setting, a  generalization  of Rademacher theorem and introduce the concept of Clarke  generalized covariant derivative to this  new context.    A comprehensive study of nonsmooth analysis in  linear context can be found in \cite{Clarke1990}.  We begin with the definition of  locally Lipschitz continuous vector fields, this  concept was introduced in \cite{CruzNetoLimaOliveira1998} for gradient vector fields and its extension to general vector fields can be found   in \cite[p. 241]{CanaryEpsteinMarden2006}.
\begin{definition} \label{def:ILC}
A  vector field $X$ on $M$   is said to be Lipschitz continuous on $\Omega\subset M$, if there exists  a constant $L > 0$ such that for  $p, q\in M$ and all $\gamma$ geodesic segment joining $p$ to $q$, there holds
$$ \left\| P_{\gamma,p,q}X(p) - X(q)\right\| \leq  L \,\ell(\gamma), \qquad \forall~ p, q \in \Omega.$$
Given  $p\in M$, if there exists  $\delta > 0$ such that   $X$ is   Lipschitz continuous on $B_{\delta}(p)$, then $X$ is said to be Lipschitz continuous  at $p$.  Moreover, if for all $p \in  M$, $X$ is Lipschitz continuous  at $p$, then $X$ is said to be locally  Lipschitz continuous on $M$.
\end{definition}

Let $d_{\mbox{\tiny $TM$}}$ be the Riemannian distance  on $TM$.  Let  us  define  the concept of {\it Lipschitz continuity of vector field  as a map between the metric spaces $(M, d)$ and $(TM, d_{\mbox{\tiny $TM$}})$}. The formal definition is:
\begin{definition} \label{def:LCMS}
A  vector field $X$ on $M$  is said to be metrically Lipschitz continuous  on $\Omega\subset M$,  if there exists a constant $L > 0$ such that
$$d_{\mbox{\tiny $TM$}}(X(p), X(q))\leq L\, d(p,q), \qquad \forall~ p, q \in \Omega.$$
Given  $p\in M$, if there exists  $\delta > 0$ such that   $X$ is metrically Lipschitz continuous  on $B_{\delta}(p)$, then $X$ is said to be metrically Lipschitz continuous  at $p$.  Moreover, if for all $p \in M$, $X$ is  metrically Lipschitz continuous  at $p$, then $X$ is said to be locally metric Lipschitz continuous on $M$.
\end{definition}

It is immediate from the last definition that all  metrically  Lipschitz continuous  vector fields   are continuous. In the next result, we present a relationship between two above definitions.

\begin{theorem}\label{teo:lcllm}
If $X$ is Lipschitz continuous  with constant $L>0$,  then $X$ is also metrically   Lipschitz  continuous   with constant $ \sqrt{1+L^2}$. As a consequence, if $X$ is locally Lipschitz continuous on $M$, then $X$ is also locally metric Lipschitz continuous on $M$.
\end{theorem}
\begin{proof}
Since $M$ is a complete manifold,  $\pi X(p) = p$ and  $\pi X(q) = q$,  it follows from   \eqref{eq:MBundle} that
\begin{equation} \label{eq:MBAp1}
d_{\mbox{\tiny $TM$}}(X(p), X(q))\leq \sqrt{d^2(p, q)+ \| P_{\gamma, p, q} X(p) - X(q)\|^2},   \qquad \forall~p, q\in M,
\end{equation}
where $\gamma$ is the minimal geodesic segment joining $p$ to $q$.  Considering that   $X$  is Lipschitz continuous  with constant $L>0$ from Definition~\ref{def:ILC} we have  $\left\| P_{\gamma,p,q}X(p) - X(q)\right\| \leq  L \,d(p,q)$ for all $p, q\in M$.  Hence,  inequality \eqref{eq:MBAp1} becomes
$$
d_{\mbox{\tiny $TM$}}(X(p), X(q))\leq \sqrt{1+L^2}~ d(p,q),
$$
for all $p, q\in M.$ Consequently,  by using Definition~\ref{def:LCMS} we conclude that  $X$ is metrically  Lipschitz continuous with constant $\sqrt{1+L^2}$. Therefore, the  proof of the first part is complete. The proof of the second part is similar.
\end{proof}
In the next definition, we present the notion of sets of measure zero to manifolds \cite{Lee2013, Sakai1996}.
\begin{definition}
A subset $E \subseteq M$ has measure zero in $M$ if for every smooth chart $(U, \varphi)$ for $M$, the subset $\varphi(E\cap U) \subseteq \mathbb{R}^n$ has $n$-dimensional measure zero.
\end{definition}

Let  $X$ be a locally Lipschitz continuous vector field on $M$. Throughout of the paper, ${\cal D}_X$ is the set defined by
$
 {\cal D}_X:= \{p\in M :~ X ~\mbox{is differentiable at} ~ p\}.
$
Locally Lipschitz  continuous vector fields  are in general non-differentiable, however,  they  are almost everywhere differentiable with respect to the Riemannian measure (see the concept of  Riemannian measure in \cite[p. 61]{Sakai1996}), i.e., the set $M \backslash {\cal D}_X$ has measure zero. This result  follows from  Rademacher theorem. A version of this theorem for locally Lipschitz continuous vector fields is given below.
\begin{theorem}\label{teo:rademacher}
If $X$ is  a locally Lipschitz continuous vector field  on $M$,  then $X$ is almost everywhere differentiable  on $M$.
\end{theorem}
\begin{proof}
Since $M$ is a $n$-dimensional smooth manifold then $TM$ is $2n$-dimensional smooth manifold.  First note that Theorem~\ref{teo:lcllm} implies that  $X$ is continuous.  Let $(U, \varphi)$  and $(W, \psi)$ be smooth charts for $M$ and $TM$, respectively,  such that $X(U)\subseteq W$ and consider the composite  mapping $\psi \circ X \circ \varphi^{-1}: \varphi(U) \to \psi (W)$.   We proceed to prove that the mapping $\psi \circ X \circ \varphi^{-1}$ is locally Lipschitz continuous on $\varphi(U)$. According to \cite[Proposition~6.10]{Loring2011} we have that the coordinate mappings  $\varphi^{-1}: \varphi(U) \to U$  and $\psi: W \to \psi(W)$ are diffeomorphisms and, in particular, continuously differentiable. Take $z\in \varphi(U)$ and $\rho > 0$ such that $B_{\rho}[z] \subset \varphi(U)$.  Since $B_{\rho}[z]$ is compact and the derivative of $\varphi^{-1}$ is continuous in $B_{\rho}[z]$, from Mean Value Inequality, see \cite [Theorem~2.14]{Azagra2005}, there exists $L_1 > 0$ such that
$
d(\varphi^{-1}(x), \varphi^{-1}(y)) \leq  L_1 \, \hat{d}(x,y),
$
for all $x, y \in B_{\rho}[z]$, where $\hat{d}$ is the Euclidean distance in $\mathbb{R}^{n}$. On the other hand, Theorem~\ref{teo:lcllm} implies that   $X$ is locally metric Lipschitz continuous on $\varphi(U)$, then shrinking $\rho>0$ if necessary,  we conclude that there exists $L_2>0$ such that
$
d_{\mbox{\tiny $TM$}}\left( X \circ \varphi^{-1}(x),  X \circ \varphi^{-1}(y)\right)\leq  L_2\,  d\left(\varphi^{-1}(x), \varphi^{-1}(y)\right),
$
for all $x, y\in B_{\rho}[z].$ Since $X(\varphi^{-1}(B_{\rho}[z]))$ is compact and the derivative of $\psi$ is continuous in $X(\varphi^{-1}(B_{\rho}[z]))$ again using Mean Value Inequality, see \cite[Theorem~2.14]{Azagra2005}, there exists $L_3>0$ such that
$$
\tilde{d} (\psi \circ X \circ \varphi^{-1}(x) , \psi \circ X \circ \varphi^{-1}(y)) \leq L_3\,  d_{\mbox{\tiny $TM$}}( X \circ \varphi^{-1}(x),  X \circ \varphi^{-1}(y)), \qquad \forall~x, y\in B_{\rho}[z],
$$
where $\tilde{d}$ is the Euclidean distance in $\mathbb{R}^{2n}$. Combining the three last inequalities we obtain that
$$
\tilde{d} (\psi \circ X \circ \varphi^{-1}(x) , \psi \circ X \circ \varphi^{-1}(y)) \leq \widetilde{L}\, \hat{d}(x,y),  \qquad \forall~x, y\in B_{\rho}[z],
$$
where $\widetilde{L} = L_1 L_2 L_3 > 0$. Hence, $\psi \circ X \circ \varphi^{-1}$ is  locally Lipschitz continuous on $\varphi(U) \subseteq  \mathbb{R}^{n}$. Therefore, from  Rademacher theorem, see \cite[Theorem~2, p. 81]{Evans1992}, we have that $\psi \circ X \circ \varphi^{-1}$ is almost everywhere differentiable on $\varphi(U)$.  Since the   charts $(U, \varphi)$  and $(W, \psi)$ are arbitrary, we conclude that $X$ is almost everywhere differentiable on $M$.
\end{proof}
Now, we introduce the {\it Clarke generalized covariant derivative} of a locally Lipschitz continuous vector field and explore some of its properties. For a comprehensive study about Clarke  generalized Jacobian in linear space, see \cite{Clarke1990}.
\begin{definition}\label{def:general}
The Clarke generalized covariant derivative of a locally Lipschitz continuous vector field $X$ is a set-valued mapping $\partial X: M  \rightrightarrows TM$ defined as
\begin{equation*}\label{dergen}
\partial X(p) :=  \emph{co} \left\{ H\in {\mathcal{L}}(T_pM):~ \exists\, \{p_k\}\subset {\cal D}_X,~ \lim_{k\to +\infty}p_k =p,~ H = \lim_{k\rightarrow +\infty}P_{p_kp}\nabla X(p_k) \right\},
\end{equation*}
where $\emph{``co"}$ represents the convex hull and ${\mathcal{L}}(T_pM)$ denotes the vector space consisting of all linear operator from $T_pM$ to $T_pM$.
\end{definition}
From Definition \ref{def:general} and  \cite[Corollary 3.1]{FernandesFerreiraYuan2017}, it is clear that if $X$ is differentiable near $p$, and its covariant derivative is continuous at $p$, then $\partial X(p) = \{\nabla X(p)\}$. Otherwise, $\partial X(p)$ could contain other elements different from $\nabla X(p)$, even if $X$ is differentiable at $p$, see \cite[Example 2.2.3]{Clarke1990}. In the following proposition, we shall show important results of the Clarke generalized covariant derivative. In particular, that the set $\partial X(p)$ is nonempty for all $p\in M$, and that $\partial X$ is locally bounded and closed, which is a generalization of \cite[Proposition 2.6.2, items (a), (b) and (c)]{Clarke1990}. These results will be very useful throughout this paper. Similar results have already been extended to functions defined in $M$, see \cite[Theorem 2.9]{Hosseini2011}.
\begin{proposition}\label{propcov}
Let $X$ be locally Lipschitz continuous vector field on $M$. The following statements are valid for any $p\in M$:
\begin{itemize}
\item [(i)] $\partial X(p)$ is a nonempty, convex and compact subset of  ${\mathcal{L}}(T_pM)$;
\item [(ii)] the set-valued mapping $\partial X:  M \rightrightarrows TM$ is locally  bounded, i.e., for all $\delta >0$, there exists a
$L>0$ such that for all $q\in B_\delta(p)$ and $V \in \partial X(q)$,  there holds $\|V\| \leq L$;
\item [(iii)] the  mapping $\partial X$ is upper semicontinuous at $p$, i.e.,  for every scalar $\epsilon > 0$, there exists a $0<\delta < r_p$,  and such that for all $q\in B_\delta (p)$,
$$
P_{qp} \partial X(q)\subset \partial X(p) + B_{\epsilon}(0),
$$
where $B_{\epsilon}(0):=\{v\in T_pM :~ \|v\| < \epsilon\}$.
Consequently,  the  set-valued mapping $\partial X$ is closed at $p$, i.e., if $\lim_{k\to +\infty}p_k=p$, $V_k \in \partial X(p_k)$, for all $k=0, 1, \ldots$,  and $\lim_{k\to +\infty} P_{p_{k} p}V_k =V$, then $V \in \partial X(p)$.
\end{itemize}
\end{proposition}
\begin{proof}
To prove item $(i)$  we define the auxiliary set
\begin{equation*}
\partial_B X(p)  := \left\{ H\in {\mathcal{L}}(T_pM):~ \exists\, \{p_k\}\subset {\cal D}_X,~ \lim_{k\to +\infty}p_k =p,~ H = \lim_{k\rightarrow +\infty}P_{p_kp}\nabla X(p_k) \right\}.
\end{equation*}
Owing to the fact that $T_pM$ is finite dimensional and $\partial X(p)$ is the  convex hull in  ${\mathcal{L}}(T_pM)$ of the set $\partial_B X(p)$, thus $\partial X(p)$ must be  convex.  Our next goal is to prove that $\partial X(p)$ is compact.  Due to  the convex hull of a compact set be  compact, it is sufficient to prove that $\partial_B X(p)$ is bounded and closed. Our first task is to prove that  $\partial_B X(p)$ is bounded. For this, take $p\in {\cal D}_X$ and $v\in T_{p}M$. Since $\nabla X(p)v = \nabla X(p,v)$ using \eqref{derdir}, the fact that $X$ is locally Lipschitz continuous on $M$ and the definition of exponential mapping we obtain that
\begin{equation*}
\left\| \nabla X(p)v\right\| = \lim_{t\rightarrow 0^+}  \left\|\dfrac{1}{t} \left[P_{\exp_p(tv)p}X(\exp_p(tv)) - X(p)\right] \right\| \leq L \| v \|,
\end{equation*}
where $L>0$ is the Lipschitz constant of $X$ around $p$. Hence, from Definition~\ref{de:normmult} we conclude that    $\left\| \nabla X(p)\right \| \leq L$, which implies that  $\partial_B X(p)$ is bounded.  To prove that  $\partial_B X(p)$ is closed,   let  $\{H_\ell\}$ be a sequence in $\partial_B X(p)$ such that $\lim_{\ell\rightarrow +\infty} H_\ell = H$.  Since $\{H_\ell\} \subset\partial_B X(p)$, there is a sequence $\{p_{k, \ell}\}$ such that  $\lim_{k \rightarrow +\infty} p_{k, \ell} = p$, and
$
\lim_{k \rightarrow +\infty}P_{p_{k, \ell} p} \nabla X(p_{k, \ell}) = H_\ell,
$
for each fixed  $\ell$. Therefore, $\lim_{k\rightarrow +\infty} p_{k, k} = p$ and $\lim_{k\rightarrow +\infty} P_{p_{k, k}p}\nabla X(p_{k, k}) = H$, and then  $H \in \partial_B X(p)$. Consequently $\partial_B X(p)$ is compact set.   To prove that   $\partial X(p)$ is a nonempty set,  we first note  that  Theorem~\ref{teo:rademacher}  implies that $X$ is almost everywhere differentiable on $M$, i.e., the set $M \backslash {\cal D}_X$ has measure zero. According to \cite[Proposition 6.8]{Lee2013}, ${\cal D}_X$ is dense in $M$. Then,  for any fixed point $p\in M$ there exists $\{p_k\}\subset {\cal D}_X $ that converges to $p$. Since the covariant derivative of $X$ at points differentiable near $p$ are bounded in norm by the Lipschitz constant and the parallel transport is an isometry we conclude that $\{P_{p_kp} \nabla X(p_k)\}$ must have at least one accumulation point and thus $\partial X(p)$ is indeed a nonempty set. To prove item $(ii)$, take  $\delta >0$, $p\in M$ and  $L>0$  the Lipschitz constant of $X$ around $p$.  The same argument used to prove item $(i)$ shows that $\left\|\nabla X(\bar{q})\right\| \leq L$ for all $\bar{q}\in B_\delta(p)\cap {\cal D}_X$. Let  $q\in B_\delta(p) $ and  $V\in \partial X(q)$. Then, there exist $H_1, \ldots, H_m \in  \partial_B  X(q)$   and $ \alpha_1, \ldots,  \alpha_m\in [0,  1]$  such that $V = \sum_{\ell=1}^{m}\alpha_\ell H_\ell$ and   $\sum_{\ell=1}^{m} \alpha_\ell = 1$.  Since $H_1, \ldots, H_m \in  \partial_B  X(q)$  then, there exists $\{q_{k, \ell}\} \subset  B_\delta(p)\cap {\cal D}_X$ with $\lim_{k \rightarrow +\infty} q_{k, \ell} = q$ such that $V = \sum_{\ell=1}^{m}\alpha_\ell\lim_{k\rightarrow +\infty}P_{q_{k, \ell}q}\nabla X(q_{k, \ell})$.   Owing to $\{q_{k, \ell}\} \subset  B_\delta(p)\cap {\cal D}_X$  we have  $\left\|\nabla X(q_{k, \ell})  \right\| \leq L$. Therefore, using that the parallel transport is an isometry we conclude that
\begin{equation*}
\left\|V\right\| = \left\| \sum_{\ell=1}^{m} \alpha_\ell \lim_{k\rightarrow +\infty} P_{q_{k, \ell}q}\nabla X(q_{k, \ell}) \right\| \leq \sum_{\ell=1}^{m} \alpha_\ell  \lim_{k \rightarrow +\infty} \left\|P_{q_{k, \ell}q}\nabla X(q_{k, \ell})  \right\| \leq L,
\end{equation*}
which is the desired inequality.  To prove item $(iii)$, suppose by contradiction that, for a given $\epsilon > 0$ and all $0 <\delta < r_p$ there exists $q\in B_{\delta}(p)$ such that $P_{qp} \partial X(q)\not\subset \partial X(p) + B_{\epsilon}(0)$. Then, there exists a sequence $\{q_k\} \subset {\cal D}_X$ such that $\lim_{k \to +\infty} q_k=p$ and
$
P_{q_k p} \nabla X(q_k) \notin \partial X(p) + B_{\epsilon}(0).
$
On the other hand, the item  $(ii)$ implies that $\partial X$ is locally bounded. Since the parallel transport is an isometry we have $\{P_{q_kp}\nabla X(q_k)\}$ is  bounded. Thus, we can extract $\{P_{q_{k_\ell}p}\nabla X(q_{k_\ell})\}$ a convergent subsequence of $\{P_{q_kp}\nabla X(q_k)\}$, let us say that $\{P_{q_{k_\ell}p}\nabla X(q_{k_\ell})\}$ converges to some $H$. From  Definition \ref{def:general} we obtain that  $H \in \partial X(p)$, which is a contradiction. Therefore, $\partial X$ is upper semicontinuous at $p$. The last part of the  item $(iii)$  is an immediate consequence of the first part and the proof is complete.
\end{proof}
\section{The Newton Method} \label{sec:nmca}
In this section, we present the Newton method for finding a singularity of a vector field $X$ on $M$, i.e., to solve the following problem
\begin{equation} \label{eq:TheProblem}
\mbox {find}  \quad p\in M\quad  \mbox{such that} \quad X(p)=0,
\end{equation}
where  $X$ is  a locally Lipschitz  continuous vector field on $M$. We will  study the   local and semi-local properties of the sequence generated by the method.  In the following, we formally state the Newton method to solve the problem \eqref{eq:TheProblem}.\\

\hrule
\begin{algorithm}  \label{Alg:NNM}
{\bf Newton  method\\}
\hrule
\begin{description}
\item[\bf Step 0.] Let $p_0\in M$ be given, and set $k=0$.
\item[\bf Step 1.] If $X(p_k) = 0$, stop.
\item[\bf Step 2.] Choose a $V_k\in \partial X(p_k)$ and compute
\begin{equation} \label{eq:NM}
p_{k+1}=\exp_{p_{k}}(-V_{k}^{-1}X(p_{k})).
\end{equation}
\item[\bf Step 3.] Set $k\gets k+1$, and go to step~1.
\end{description}
\hrule
\end{algorithm}
This method is a natural extension to the Riemannian setting of Newton method introduced in \cite{Qi1993}. Note that for to guarantee the well-definedness of the method, there are two issue which deserve attention  in each iteration $k$. The Clarke generalized covariant derivative  $\partial X(p_k)$ must be nonempty, which has already been proven in  the item $(i)$ of Proposition~\ref{propcov} and  $V_k \in \partial X(p_k)$ must be non-singular. In the following section, we will study the well-definedness  and  convergence properties of Newton method.
\subsection{Local Convergence Analysis}
In this section, we present the local convergence  analysis of Algorithm~\ref{Alg:NNM}.  For this end, we assume that $p_*\in M$ is a solution of problem \eqref{eq:TheProblem}. First, we will show that under some assumptions, the sequence generated by   this algorithm starting from  a suitable neighborhood of $p_*$ is well-defined  and converges to $p_*$ with rate of order $1 + \mu$. We begin by introducing the concept of regularity.
\begin{definition}\label{def:reg}
We say that a vector field $X$  on $M$ is regular at $p \in M$ if all $V_p \in \partial X(p)$ are non-singular.  If $X$ is regular at every point of $\Omega \subseteq M$, we say that $X$ is regular on $\Omega$.
\end{definition}
In the  following,   we study the behavior of the Newton method for a special class of vector field in a neighborhood of a regular point.  For this,   assume that {\it  $X$ is  a locally Lipschitz  continuous vector field on $M$}.  Consider the   following condition:
\begin{itemize}
\item[{\bf A1.}]  Let  ${\bar p}\in M$, $0<\delta<  r_{\bar p}$,   $X$ be regular on $B_{\delta}({\bar p})$,  $\lambda_{\bar p}\geq \max\{\|{V_{\bar p}^{-1}}\|: ~ { V_{\bar p}}\in \partial X({\bar p})\}$  and   $\epsilon >0$ with $\epsilon \lambda_{\bar p}<1$. For all $p, q\in   B_{\delta}({\bar p})$ and $V_{p}\in \partial X(p)$ there hold
\begin{eqnarray}
                                                                                                \displaystyle \|V_{p}^{-1}\| &\leq&  \frac{\lambda_{\bar p}}{1 -\epsilon \lambda_{\bar p}},   \label{eq:fcA1}    \\
     \displaystyle \left\|X(q)-P_{pq}\left[X(p)+ V_{p}\exp^{-1}_{p} q\right]\right\|&\leq& \epsilon \,d(p, q)^{1+\mu}, \qquad 0\leq \mu \leq 1. \label{eq:scA1}
\end{eqnarray}
\end{itemize}

The above  assumption  guarantee, in particular, that $X$ is regular  in a neighborhood of ${\bar p}$ and consequently, the Newton iteration  mapping is well-defined. Let  $0< \delta< r_{\bar p}$ be given by above assumption and  $N_{X} \colon B_{\delta}({\bar p}) \rightrightarrows M$ be  the {\it Newton iteration  mapping} for $X$ defined by
$$
N_{X} (p):=\left\{  \exp_{p}(-V_p^{-1}X(p))~: ~ V_p\in \partial X(p) \right\}.
$$

Therefore, one can apply a single Newton iteration on any $p \in B_{\delta}({\bar p})$ to obtain $N_X(p)$, which may  be  not  included  in $B_{\delta}({\bar p})$. Thus, this is enough to guarantee the well-definedness of only one iteration. In the  next result, we establish that Newtonian iterations may be repeated indefinitely in a suitable neighbourhood of  ${\bar p}$.
\begin{lemma}\label{le:welldefined}
Suppose  that $p_* \in M$ is  a solution of problem~\eqref{eq:TheProblem}, $X$  satisfies {\bf A1}  with ${\bar p} = p_*$, $q = p_*$ and the constants  $\epsilon>0$, $0<\delta<  r_{p_*}$ and $0\leq \mu \leq 1$  satisfy  $\epsilon \lambda_{p_*}(1+ \delta^{\mu}K_{p_*}) <1$.  Then, there exists   ${\hat \delta}> 0$ such that $X$ is regular on $ B_{{\hat \delta}}(p_*)$ and
\begin{equation} \label{eq;csl}
d\left(\exp_{p}(-V_{p}^{-1}X(p)),p_{*}\right)\leq \frac{ \epsilon\lambda_{p_*}K_{p_*}}{1 - \epsilon \lambda_{p_*} }  d(p,p_{*})^{1+\mu},   \qquad \forall~p \in B_{\hat \delta}({p_*}),  \qquad \forall ~V_{p}\in \partial X(p).
\end{equation}
Consequently,  $N_{X}$ is well-defined on $B_{{\hat \delta}}(p_{*})$ and   $N_{X}(p) \subset   B_{{\hat \delta}}(p_{*})$   for all $p\in B_{\hat \delta}(p_{*})$.
\end{lemma}
\begin{proof}
Since $\epsilon>0$, $0<\delta<  r_{p_*}$ and $0\leq \mu \leq 1$  are constants   such that  $X$ satisfies {\bf A1},   $X(p_*) = 0$ and  the  parallel transport is an isometry,  we conclude that
\begin{equation}
\begin{split} \label{equa8}
\left\| V_p^{-1} X(p) + \exp_p^{-1}p_* \right\| &\leq \left\|V_p^{-1}\right\|  \left\|X({p_*})-P_{p{p_*}}\left[X(p)+ V_{p}\exp^{-1}_{p} {p_*}\right]\right\| \\
                                                                      &\leq \frac{ \epsilon\lambda_{p_*}}{1 - \epsilon \lambda_{p_*}}  d(p,p_{*})^{1+\mu},
\end{split}
\end{equation}
for all $p \in B_{\delta}(p_*)$ and $V_{p}\in \partial X(p)$.  Hence,   \eqref{equa8}  implies that  there exists  $0 < {\hat \delta}< \delta $ such that    $\|V_p^{-1} X(p) + \exp_p^{-1}p_* \| \leq r_{p_*}$ for all $p \in B_{\hat \delta}(p_*)$ and $V_{p}\in \partial X(p)$.   Thus, considering that $\|\exp_p^{-1}p_*\|=d(p, p_*)<  r_{p_*}$,   we can use Definition~\ref{def:kp} with $p=p_*$, $q=p$, $u=-V_{p}^{-1}X(p)$ and  $v= \exp_p^{-1}p_*$ to obtain that
$$
d\left(\exp_{p}(-V_{p}^{-1}X(p)), p_*\right) \leq K_{p_*}\left\|-V_p^{-1} X(p) - \exp_p^{-1}p_*\right\|,
$$
for all $p \in B_{  \hat \delta}(p_*)$ and $V_{p}\in \partial X(p)$.
Therefore,  the combination of the last inequality with    \eqref{equa8} yields   \eqref{eq;csl}.  Owing to  $0 < {\hat \delta}< \delta $ and   $X$ be regular on $ B_{\delta}(p_*)$, we conclude that  $N_{X}$ is well-defined on $B_{\hat \delta}(p_{*})$.  Moreover,  since $\epsilon\lambda_{p_*}(1+ \delta^{\mu}K_{p_*}) <1$ and $0< {\hat \delta}<\delta$  we have from \eqref{eq;csl} that  $d\left(\exp_{p}(-V_{p}^{-1}X(p)),p_{*}\right)< d(p,p_{*})$  for all $p \in B_{\hat \delta}({p_*})$ and  $V_{p}\in \partial X(p)$. Thus,  we obtain   that   $N_{X}(p) \subset   B_{\hat \delta}(p_{*})$   for all $p\in B_{\hat \delta}(p_{*})$, and the proof of the lemma is complete.
\end{proof}
Now, we are ready to establishes the main result of this section, its prove is a  straight application of Lemma~\ref{le:welldefined}.

\begin{theorem}\label{th:conv}
Suppose   that  $p_* \in M$ is  a solution of problem~\eqref{eq:TheProblem}, $X$  satisfies {\bf A1}  with ${\bar p} = p_*$, $q = p_*$ and the constants  $\epsilon>0$, $0<\delta<  r_{p_*}$ and $0\leq \mu \leq 1$  satisfy  $\epsilon \lambda_{p_*}(1+ \delta^{\mu}K_{p_*}) <1$. Then, there exists   $0 < {\hat \delta} < \delta$ such that   for each  $p_{0}\in B_{\hat \delta}(p_{*})\backslash \{{p_*}\}$,    $\{p_{k}\}$ in Algorithm~\ref{Alg:NNM} is well-defined, belongs to  $B_{\hat \delta}(p_{*})$ and converges to $p_{*}$ with order $1+\mu$ as follows
\begin{equation} \label{eq;csls}
d\left(p_{k+1},p_{*}\right)\leq \frac{ \epsilon\lambda_{p_*}K_{p_*}}{1 - \epsilon \lambda_{p_*}}  d(p_{k},p_{*})^{1+\mu},  \qquad k=0, 1, \ldots.
\end{equation}
\end{theorem}
\begin{proof}
The definition of Newton mapping $N_{X}$  implies that the sequence  generated by Algorithm~\ref{Alg:NNM} is equivalently stated as
\begin{equation}\label{equ10}
p_{k+1} \in  N_X(p_k), \qquad k = 0, 1, \ldots.
\end{equation}
Thus, by using \eqref{equ10},   we can apply   Lemma~\ref{le:welldefined}  to conclude  that  there exists  $0 < {\hat \delta} < \delta$  such that if   $p_{0}\in B_{\hat \delta}(p_{*})\backslash \{{p_*}\}$,  then   $\{p_{k}\}$ in Algorithm~\ref{Alg:NNM} is well-defined,  belongs to  $B_{\hat \delta}(p_{*})$ and satisfies \eqref{eq;csls}.  Since $\{p_{k}\}$ belongs to  $B_{\hat \delta}(p_{*})$ and  $\epsilon \lambda_{p_*}(1+\delta^{\mu}K_{p_*}) <1$   we obtain from   \eqref{eq;csls} that
$$
d\left(p_{k+1},p_{*}\right) <   \frac{ \epsilon \lambda_{p_*}{\hat \delta}^{\mu} K_{p_*}}{1 - \epsilon \lambda_{p_*}}d(p_{k},p_{*})< d(p_{k},p_{*}),  \qquad k=0, 1, \ldots.
$$
Therefore, we  conclude that $\{p_k\}$ converges to $p_*$ with order $1+\mu$ as \eqref{eq;csls}.
\end{proof}
\begin{remark}
If $\mu = 0$ in Theorem~\ref{th:conv}, then   \eqref{eq;csls} holds for any $\epsilon > 0$ satisfying $\epsilon \lambda_{p_*}(1+ K_{p_*}) <1$, independently  of ${\hat \delta}$. Thus,    \eqref{eq;csls}  implies  that $\{p_k\}$ converges superlinearly to $p_*$.
\end{remark}
\subsubsection{Local Convergence for  Semismooth Vector Fields}
In this section, we present a local convergence theorem for the Newton  method  for  finding a singularity of  semismooth vector fields. Semismoothness in  Euclidean setting    was originally introduced by Mifflin \cite{Mifflin1977} for scalar-valued functions  and subsequently extended by Qi and Sun in \cite{Qi1993} for vector-valued functions. The extension of   semismoothness  to the Riemannian settings was presented in \cite{LagemanHelmekeManton},  and  it will plays an important role in this section.   As occur in Euclidean context,   semismooth vector fields  are in general nonsmooth.  However,   as we shall show the Newton method is still applicable and converges locally with superlinear rate to a regular solution. Before state  formally the concept of semismoothness in the Riemannian setting,  let us first to show that  locally Lipschitz continuous vector fields are regular near regular points. The statement  of the result is:
\begin{lemma}\label{le:NonSing}
Let $X$ be a locally Lipschitz  continuous vector field on $M$. Assume that $X$  is   regular  at ${p_*} \in M$ and let $\lambda_{p_*}\geq \max\{\|{V_{p_*}^{-1}}\|: ~ {V_{p_*}}\in \partial X({p_*})\}$. Then, for every    $\epsilon>0$  satisfying  $\epsilon \lambda_{p_*} <1$, there exists $0 < \delta < r_{p_*}$  such that $X$ is regular on $ B_{\delta}(p_*)$ and
\begin{equation} \label{eq:BanachLem}
\|V_p^{-1}\| \leq  \frac{\lambda_{p_*}}{1 -\epsilon \lambda_{p_*}}, \qquad  \forall~p\in   B_{\delta}({p_*}), \quad \forall ~V_p\in \partial X(p).
\end{equation}
\end{lemma}
\begin{proof}
Let $\epsilon>0$ such that $\epsilon \lambda_{p_*} <1$.  Since $X$ is  a  locally Lipschitz  continuous vector field, it follows from item~$(iii)$ of Proposition~\ref{propcov} that, there is  a $0<\delta < r_{p_*}$ such that for all $p\in B_\delta ({p_*})$,
$P_{p{p_*}} \partial X({p})\subset \partial X({p_*}) + \{V\in T_{p_*}M :~ \|V\| < \epsilon\}$, i.e.,
$$
\partial X(p) \subset \left\{V \in T_{p}M: ~ \|P_{p{p_*}}V-{ V_{p_*}}\|< \epsilon, ~  \mbox{for some}~   { V_{p_*}} \in \partial X({p_*})\right\},  \qquad  \forall~p\in B_\delta ({p_*}).
$$
This inclusion implies that,   for each $p\in B_\delta ({p_*})$ and  $V_{p}\in \partial X(p)$,   there exists ${ V_{p_*}} \in \partial X({p_*})$ non-singular such that  $\|V_{p_*}^{-1}\| \|P_{p{p_*}}V_{p}-{ V_{p_*}}\|<\epsilon \lambda_{p_*}< 1$. Thus,  taking into account that the parallel transport is an isometry,  it follows from Lemma~\ref{le:Banach}   that $V_{p}$ is non-singular and
$$
\|V_{p}^{-1}\| \leq \frac{\| V_{p_*}^{-1}\|}{1 - \|V_{p_*}^{-1}\|\|P_{p{p_*}}V_{p}-{ V_{p_*}}\|}.
$$
Therefore,  considering that $\|V_{p_*}^{-1}\|\leq \lambda_{p_*}$  and $\|P_{p{p_*}}V_{p}-{ V_{p_*}}\|< \epsilon$,  the inequality \eqref{eq:BanachLem} follows.
\end{proof}
Now, let us  present a class of vector fields  satisfying the assumption   {\bf A1}, namely  the  semismooth vector fields and $\mu$-order semismooth vector fields. There exist, in the Euclidean context, several equivalents definitions  of the concept of    semismoothness,  see \cite{Qi1993}, see also \cite[Definition 7.4.2, p.~677]{Facchinei2003}. In the present paper, we  will  extend to the Riemannian settings the concept of  semismoothness adopted  in  \cite[p. 411]{DontchevRockafellar2010Book}.
\begin{definition} \label{def:DefSS}
A vector field $X$ on $M$,  which is Lipschitz continuous at ${p_*}$ and  directionally differentiable at  $p\in B_{\delta}({p_*})$ for all direction at $T_{p}M$,   is said to be semismooth at ${p_*} \in M$  when  for every  $\epsilon>0$  there exists $0<\delta  < r_{p_*}$ such that
$$
\displaystyle  \left\|X({p_*})-P_{p{p_*}}\left[X(p)+ V_{p}\exp^{-1}_{p} {p_*}\right]\right\|\leq  \epsilon \,d(p, {p_*}), \qquad   \forall~p \in B_{\delta}({p_*}), \quad \forall~ V_{p}\in \partial X(p).
$$
The  vector filed $X$ is said to be $\mu$-order semismooth at ${p_*} \in M$, for $0< \mu \leq 1$, when  there exist  $\epsilon>0$   and  $0<\delta  < r_{p_*}$ such that
\begin{equation} \label{eq:MuOrder}
 \displaystyle  \left\|X({p_*})-P_{p{p_*}}\left[X(p)+ V_{p}\exp^{-1}_{p} {p_*}\right]\right\| \leq   \epsilon \,d(p, {p_*})^{1+\mu}, \qquad   \forall~p \in B_{\delta}({p_*}), \quad \forall~ V_{p}\in \partial X(p).
\end{equation}
\end{definition}
Next, we state and prove the local  convergence result  for the   Newton  method  for  finding a singularity of  semismooth and $\mu$-order semismooth  vector fields.
\begin{theorem}\label{th:mainss}
Let  $X$ be locally Lipschitz continuous vector filed on $M$ and $p_{*}\in M$ be  a solution of problem~\eqref{eq:TheProblem}. Assume that $X$ is semismooth and   regular at ${p_*}$.   Then, there exists a $\delta>0$  such that for each  $p_{0}\in B_{\delta}(p_{*})\backslash \{{p_*}\}$,   $\{p_k\}$  generated by Algorithm~\ref{Alg:NNM}, is well-defined, belongs to  $B_{\delta}(p_{*})$ and converges superlinearly to $p_{*}$. Additionally, if $X$ is $\mu$-order semismooth  at $p_*$, then the convergence of  $\{p_k\}$ to    $p_{*}$   is of order $1+\mu$.
\end{theorem}
\begin{proof}
Owing to  $X$ be semismooth and  regular at ${p_*}\in M$,  we can take $\lambda_{p_*}\geq \max\{\|{V_{p_*}^{-1}}\|: ~ {V_{p_*}}\in \partial X({p_*})\}$.  Take  $\epsilon>0$  satisfying    $\epsilon \lambda_{p_*}(1+ K_{p_*}) <1$. Thus,   from Lemma~\ref{le:NonSing} and Definition~\ref{def:DefSS} we can take $\delta>0$ such that \eqref{eq:fcA1} and \eqref{eq:scA1} hold for $\mu = 0$.   Hence, assumption $\mathbf{A1}$ holds with ${\bar p}=p_*$ and $q = p_*$  for all $p\in B_\delta({p_*})$ and $\mu = 0$.   Therefore,  applying  Theorem~\ref{th:conv},  we obtain that there exists   $0 < {\hat \delta}< \delta$ such that  every sequence $\{p_k\}$  generated by Algorithm~\ref{Alg:NNM} with $p_0 \in  B_{\hat \delta}({p_*})\backslash \{{p_*}\}$  belongs to $B_{\hat \delta}({p_*})$ and  satisfies \eqref{eq;csls}. Hence,  we have
\begin{equation*} \label{eq:BoundNormSL}
\frac{d(p_{k+1},p_{*})}{d(p_{k},p_{*})}  \leq \frac{ \epsilon \lambda_{p_*}K_{p_*}}{1 - \epsilon \lambda_{p_*} } ,   \qquad k = 0, 1, \ldots.
\end{equation*}
Since the last equality holds for any  $\epsilon$ such that $0<\epsilon< 1/( \lambda_{p_*}(1+K_{p_*})) $,  we conclude that  $\{p_k\}$ converges superlinearly to ${p_*}$. The proof of the second part  is similar. Indeed,  for a given  $\epsilon >0$   with   $\epsilon \lambda_{p_*} <1$,  take $\delta  > 0$ satisfying  $\epsilon \lambda_{p_*}(1+ \delta^{\mu}K_{p_*}) <1$  and such that \eqref{eq:BanachLem} and   \eqref{eq:MuOrder} hold. Then,  we can apply   Theorem~\ref{th:conv} and  the proof follows.
\end{proof}
In the following  remark, in particular,  we show that with some adjustments Theorem~\ref{th:mainss}  reduces to some well-known results.
\begin{remark}
It is well-known that  Newton method and its variants are quite efficient for finding zero on nonlinear functions in Euclidean settings. This because they have excellent convergence rate in a neighborhood of a zero.   It was shown  in  \cite{Qi1993}  that for some  class of nonsmooth functions,   namely    for semismooth functions,  the convergence of  Newton method still  is guaranteed.  The above theorem, allows us conclude that the generalization of  Newton method from the linear context to Riemannian settings for finding singularities of semismooth vector fields  still preserves its main convergence properties. It is worth mentioning that if $X$  is  continuously differentiable,  then Theorem~\ref{th:mainss} reduces to   \cite[Theorem~3.1]{FernandesFerreiraYuan2017}.   If $M = \mathbb{R}^n$, then Theorem~\ref{th:mainss} reduces to first part of  \cite[Theorem~3.2]{Qi1993}, see also  \cite[Theorem~7.5.3, p. 693]{Facchinei2003}. Finally,  if $X$  is  continuously differentiable  and $M = \mathbb{R}^n$, then the theorem above reduces to the first part of  \cite[Proposition~1.4.1, p. 90]{Bertsekas2016}.
\end{remark}
\subsection{Semi-local Convergence Analysis}
In this  section, we state and prove  the Kantorovich-type theorem on Newton method. This theorem ensures that  the sequence generated by the method  converges towards a singularity of the vector field by using semi-local conditions. It  is worth noting that the theorem   does not require a priori existence of a singularity, proving instead the existence of the singularity and its uniqueness on some region. The statement of the theorem is:
\begin{theorem} \label{th:KantTh}
Let  $X$ be locally Lipschitz continuous vector filed on $M$ and    $p_0\in M$.   Suppose that   $X$  satisfies {\bf A1}  with ${\bar p} = p_0$,  $\mu=0$ and $\delta > {\bar \delta}$. Moreover,    $B_{\bar \delta}(p_0)\subset M$ is a totally normal neighborhood of $p_0$ and  the constants $\lambda_{p_0}>0$,  $\epsilon>0$ and $0<{\bar \delta}<  r_{p_0}$ are such that
\begin{equation}\label{eq:Ass1}
\epsilon \lambda_{p_0} <\frac{1}{2}, \qquad  \qquad  \frac{\lambda_{p_0}  }{1 -2\epsilon \lambda_{p_0}} \left\| X(p_0)\right\|\leq  {\bar \delta}.
\end{equation}
Then,   $\{p_{k}\}$ in Algorithm~\ref{Alg:NNM} is well-defined, belongs to $B_ {\bar \delta}(p_0)$ and converge towards the unique solution $p_*$ of problem \eqref{eq:TheProblem} in $B_ {\bar \delta}[p_0]$. Furthermore, the  following error estimate  holds
\begin{equation}\label{erro}
d(p_k, p_*) \leq \frac{\epsilon \lambda_{p_0}}{1-2\epsilon \lambda_{p_0}}\,d(p_k,p_{k-1}), \qquad k = 1,2, \ldots.
\end{equation}
\end{theorem}
\begin{proof}
Firstly, we will prove by induction that   the sequence $\{p_{k}\}$ in Algorithm~\ref{Alg:NNM} is well-defined,  belong to $B_{\bar \delta}(p_0)$ and  satisfies
\begin{equation}\label{eq:ikt}
d(p_{k+1}, p_k)  \leq   \left(\frac{\epsilon \lambda_{p_0}}{1-\epsilon \lambda_{p_0}}\right)^k {\bar \delta} \left(\frac{1 - 2\epsilon \lambda_{p_0}}{1-\epsilon \lambda_{p_0}}\right),    \qquad k = 0, 1, \ldots.
\end{equation}
Let $V_{0} \in \partial X(p_0)$.  Assumption {\bf A1}  implies that   $ V_{0}$ is non-singular and  $\| V_{0}^{-1}\| \leq \lambda_{p_0}/(1 - \epsilon \lambda_{p_0})$. Hence   using  \eqref{eq:NM} we obtain that the iterate $p_{1}$ is well-defined. Furthermore,  the definition of the exponential mapping and \eqref{eq:Ass1}  imply that
\begin{align*}\label{equ23}
d(p_1, p_0)  = d\left(\exp_{p_0}(-V_{0}^{-1}X(p_0)), p_0\right)\leq  \left\| -V_{0}^{-1} X(p_0) \right\| \leq \frac{\lambda_{p_0}}{1 - \epsilon \lambda_{p_0}} \left\| X(p_0)\right\| \leq  {\bar \delta}\left(\frac{1 - 2\epsilon \lambda_{p_0}}{1-\epsilon \lambda_{p_0}}\right)< {\bar \delta}.
\end{align*}
Therefore,  $p_{1}$ is well-defined,  belong to $B_{\bar \delta}(p_0)$ and   \eqref{eq:ikt} holds  for  $k=0$.   Assume by induction  that $p_{1}, \ldots,   p_{\ell-1}$ are well-defined,   belong to $B_{\bar \delta}(p_0)$ and  \eqref{eq:ikt} holds for $k=1, \dots, \ell-1$.      Since $p_{\ell-1} \in B_{\bar \delta}(p_0)$ it follows from  assumption {\bf A1}  that  $V_{\ell-1}$ is non-singular and, in particular, the iterate $p_{\ell}$ is well-defined.  Thus, using triangular inequality and  the  induction assumption we have
\begin{equation} \label{eq:cbx0}
d(p_{\ell},p_0) \leq  \sum_{j=1}^{\ell} d(p_{j}, p_{j-1}) \leq   {\bar \delta}\left(\frac{1 - 2\epsilon \lambda_{p_0}}{1-\epsilon \lambda_{p_0}}\right) \sum_{j=1}^{\ell} \left(\frac{\epsilon \lambda_{p_0}}{1-\epsilon \lambda_{p_0}}\right)^{j-1}<  {\bar \delta},
\end{equation}
Then,  $ p_{\ell} \in B_{\bar \delta}(p_0)$.  Since $p_{\ell} \in B_{\bar \delta}(p_0)$ it follows from  {\bf A1}  that  $V_{\ell}$ is non-singular and, in particular, the iterate $p_{\ell+1}$ is well-defined.  Moreover,   $\| V_{\ell}^{-1}\| \leq \lambda_{p_0}/(1 - \epsilon \lambda_{p_0}) $.   Thus,  using  \eqref{eq:NM} and    the definition of the exponential mapping  we have
\begin{equation} \label{eq:pkpkk}
d(p_{\ell+1}, p_{\ell})  = d\left(\exp_{p_{\ell}}(-V_{{\ell}}^{-1}X(p_{\ell})), p_{\ell}\right)\leq\left\| -V_{{\ell}}^{-1}X(p_{\ell})\right\|\leq  \frac{\lambda_{p_0}}{1 - \epsilon \lambda_{p_0}}\left\| X(p_{\ell})\right\|.
\end{equation}
On the other hand,  considering that $B_{\bar \delta}(p_0)$ is  a totally normal neighborhood  and $p_{{\ell}-1}, p_{\ell} \in B_{\bar \delta}(p_0)$,  we conclude after some algebraic manipulations that
$$
\left\| X(p_{\ell})\right\| \leq \left\|X(p_{\ell}) - P_{p_{{\ell}-1}p_{\ell}} \left[X(p_{{\ell}-1}) + V_{{\ell}-1} \exp_{p_{{\ell}-1}}^{-1}p_{\ell}\right] \right\|
+  \left\| X(p_{{\ell}-1}) + V_{{\ell}-1} \exp_{p_{{\ell}-1}}^{-1}p_{\ell} \right\|.
$$
Taking into account that   \eqref{eq:NM} implies  $X(p_{{\ell}-1}) + V_{{\ell}-1} \exp_{p_{{\ell}-1}}^{-1}p_{\ell}=0$ the last inequality becomes
$$
\left\| X(p_{\ell})\right\| \leq \left\|X(p_{\ell}) - P_{p_{{\ell}-1}p_{\ell}} \left[X(p_{{\ell}-1}) + V_{{\ell}-1} \exp_{p_{{\ell}-1}}^{-1}p_{\ell}\right] \right\|.
$$
Using  {\bf A1} with $q = p_{\ell}$, $p = p_{{\ell}-1}$ and $V_p = V_{{\ell}-1}$, it follows from the latter  inequality   that
$$
\| X(p_{\ell})\| \leq \epsilon d(p_{\ell}, p_{{\ell}-1}).
$$
Thus combining the last inequality  with  \eqref{eq:pkpkk} and  using the  induction assumption we conclude that
\begin{equation}\label{equ25}
d(p_{{\ell}+1}, p_{\ell}) \leq  \frac{\epsilon \lambda_{p_0}}{1-\epsilon \lambda_{p_0}} d(p_{\ell}, p_{{\ell}-1}) \leq  \left(\frac{\epsilon \lambda_{p_0}}{1-\epsilon \lambda_{p_0}}\right)^{\ell} {\bar \delta} \left(\frac{1 - 2\epsilon \lambda_{p_0}}{1-\epsilon \lambda_{p_0}}\right),
\end{equation}
and the induction proof is complete.  Hence, using \eqref{equ25} and the same argument used to prove \eqref{eq:cbx0} we  obtain  that $p_{{\ell}+1} \in B_{\bar \delta}(p_0)$. Therefore, the Newton iterates are well-defined,  belong to $B_{\bar \delta}(p_0)$ and satisfy \eqref{eq:ikt}. We proceed  to prove that  $\{p_{k}\}$ converges. Using the triangular inequality and \eqref{eq:ikt}, for any $k$ and $s\in \{0, 1, \ldots\}$ we have
$$
d(p_{k+s+1}, p_k) \leq \sum_{j=k}^{k+s} d(p_{j+1}, p_j) \leq  {\bar \delta} \left(\frac{1 - 2\epsilon \lambda_{p_0}}{1-\epsilon \lambda_{p_0}}\right) \sum_{j=k}^{k+s}\left(\frac{\epsilon \lambda_{p_0}}{1-\epsilon \lambda_{p_0}}\right)^j <  {\bar \delta}\left(\frac{\epsilon \lambda_{p_0}}{1-\epsilon \lambda_{p_0}}\right)^k.
$$
Thus, due to  $2\epsilon \lambda_{p_0} <1$ we conclude that $\{p_k\}$ is a Cauchy sequence. This implies that the sequence  $\{p_k\}$ converges to some $p_* \in B_{\bar \delta}[p_0]$.  Thus,  owing  to  $X$ be a locally Lipschitz continuous vector field,  item~$(ii)$ of Proposition~\ref{propcov} implies that  $\left\{V_k\right\}$ is bounded. Therefore, using that $X$ is continuous, \eqref{eq:NM},  some properties of norm and $\{p_k\}$ converges to $p_*$,  we have
\begin{equation*}
0 \,\leq \left\| X(p_*)\right\| = \lim_{k \rightarrow +\infty} \left\| X(p_k)\right\| = \lim_{k\rightarrow +\infty} \left\|-V_k \exp_{p_k}^{-1}p_{k+1}\right \| \leq \lim_{k \rightarrow +\infty}\left\| V_k \right\|  d(p_{k+1}, p_k) = 0,
\end{equation*}
consequently, $X(p_*) = 0$. Now, we are going to  prove the uniqueness of the solution in $B_{\bar \delta}[p_0]$. For this purpose, assume that $q\in B_{\bar \delta}[p_0]$ and $X(q) = 0$. Take $V_* \in \partial X(p_*)$, by assumption $X$ is  regular on  $B_{\delta}(p_0)$ and $B_{\bar \delta}[p_0]\subset B_{\delta}(p_0)$,  then  $V_*$ is non-singular. Since $X(p_*) = 0$ and $X(q) = 0$, using  {\bf A1} with $p = p_*$ and $V_p = V_*$, and  some manipulations we obtain that
\begin{multline*}
d(p_*,q)  = \left\| V_*^{-1} V_* \exp^{-1}_{p_*}q \right\| \leq \left\| V_*^{-1}\right\|\Big[\left\|X(q) - P_{p_*q}\left[X(p_*)  + V_*\exp_{p_*}^{-1}q\right]\right\|\Big] \leq \frac{\epsilon \lambda_{p_0}}{1-\epsilon \lambda_{p_0}} d(p_*,q).
\end{multline*}
Thus, since $\epsilon \lambda_{p_0}< 1/2$ we conclude that $d(p_*,q) = 0$, i.e., $q = p_*$. Therefore, $p_*$ is the unique solution of \eqref{eq:TheProblem} in $B_{\bar \delta}[p_0]$. It remains to show \eqref{erro}. First note that using the same argument to establishes  first inequality in \eqref{equ25}   we can also prove that   $d(p_{i+1}, p_i) \leq [\epsilon \lambda_{p_0}/(1-\epsilon \lambda_{p_0})]\,d(p_i, p_{i-1})$,  for all $i=1, 2, \ldots$ and then
$$
d(p_{k+j}, p_{k+j-1}) \leq  \left(\frac{\epsilon \lambda_{p_0}}{1-\epsilon \lambda_{p_0}}\right)^{j}d(p_k, p_{k-1}),\qquad j=1, 2, \ldots.
$$
Hence,  using the triangular inequality, the last inequality and  any $s\in \{0, 1, \ldots\}$   we conclude
$$
d( p_{k+s+1}, p_k) \leq \sum_{j=k}^{k+s}d(p_{j+1}, p_j)  \leq  d(p_k, p_{k-1}) \sum_{j = 1}^{s+1} \left(\frac{\epsilon \lambda_{p_0}}{1-\epsilon \lambda_{p_0}} \right)^j   < \dfrac{\epsilon \lambda_{p_0}}{1- 2\epsilon \lambda_{p_0}}\, d(p_k, p_{k-1}).
$$
Taking the limit as $s$ goes to $+\infty$, we obtain the inequality \eqref{erro}, and the proof  is complete.
\end{proof}

\section{Some Examples}\label{sec:exam}
In this section, we present a class of  examples of locally Lipschitz continuous vector fields on the sphere satisfying the assumption \textbf{A1}. For this purpose, we begin by presenting some basic definitions about the geometry of the sphere. For further details, see \cite{FerreiraIusem2013, FerreiraIusem2014} and references therein.

Let $\langle \cdot, \cdot\rangle$ be the {\it usual inner product on $\mathbb{R}^{n+1}$}, with corresponding {\it norm} denoted by $\| \cdot\|$. The {\it $n$-dimensional Euclidean sphere} and its {\it tangent hyperplane at a point $p$} are denoted, respectively, by
$$
\mathbb{S}^{n}:=\left\{ p=(p_1, \ldots, p_{n+1}) \in \mathbb{R}^{n+1}: ~ \|p\|=1\right\}, \qquad T_{p}\mathbb{S}^n:=\left\{v\in \mathbb{R}^{n+1}:~ \langle p, v \rangle=0 \right\}.
$$
Denotes by  $I$ the $(n+1)\times (n+1)$ identity  matrix. The  {\it projection onto the tangent hyperplane} $T_p\mathbb{S}^n$ is the linear mapping defined by $I-pp^T: \mathbb{R}^{n+1} \to T_p\mathbb{S}^n$, where $p^T$ denotes the transpose of the vector $p$. Let  $\Omega $ be an open set in $\mathbb{R}^{n+1}$ such that $\mathbb{S}^n \subset \Omega $, and   $Y: \Omega \to \mathbb{R}^{n+1}$   be any  semismooth mapping; several examples can be found in \cite{DontchevRockafellar2010Book,Facchinei2003, IzmailovSolodov2014}.   Then, we define  the vector field  $X:\mathbb{S}^n \to  \mathbb{R}^{n+1}$ as follows
$$
X(p):= (I-pp^T) Y(p).
$$
Note that $X(p)\in T_p\mathbb{S}^n$ for all $p\in \mathbb{S}^n$. The Clarke generalized covariant derivative of $X$  at $p$ is  given by
\begin{equation}\label{eq:subX}
\partial X(p):= \left(I-pp^T\right)\partial Y(p) -p^TY(p)I,
\end{equation}
where $\partial Y(p)$ is the Clarke generalized covariant derivative  of $Y$ at $p$. Therefore,    all $V_p \in \partial X(p)$ is  a linear mapping
$V_p:T_p\mathbb{S}^n \to T_p\mathbb{S}^n$ given  by $V_p:=\left(I-pp^T\right) {\tilde V}_p -p^TY(p)I$,  where    ${\tilde V}_p\in \partial Y(p)$.  Since $Y$ is a locally Lipschitz continuous mapping, from  Rademacher theorem, see \cite[Theorem~2, p.~81]{Evans1992}, we conclude that $Y$ is almost everywhere differentiable. As $I - pp^T$ is a differentiable mapping, we obtain that $X$ is almost everywhere differentiable. By using the fundamental theorem of calculus in Riemannian setting  (see \cite{FerreiraSvaiter2002}), the fact that $\partial Y(p)$ is locally bounded and  continuity of  $Y$, we can prove that $X$ is also locally Lipschitz continuous vector field. Assume that $X$ is   regular  at $\bar{p} \in \Omega$ and let $\lambda_{\bar{p}}\geq \max\{\|{V_{\bar{p}}^{-1}}\|: ~ {V_{\bar{p}}}\in \partial X(\bar{p})\}$. Then, from Lemma~\ref{le:NonSing} for every $\epsilon>0$ satisfying  $\epsilon \lambda_{\bar{p}} <1$, there exists $0 < \delta < \pi$ (where $\pi$ is the injectivity radius of $\mathbb{S}^n$)  such that $X$ is regular on $ B_{\delta}(\bar{p})$ and for all $p\in   B_{\delta}(\bar{p})$ and $V_p\in \partial X(p)$ the following holds
\begin{equation*}
\|V_p^{-1}\| \leq  \frac{\lambda_{\bar{p}}}{1 -\epsilon \lambda_{\bar{p}}}.
\end{equation*}
This implies that inequality~\eqref{eq:fcA1} holds. On the other hand, because $X$ is a composition of semismooth mappings, we conclude that $X$ is semismooth, see \cite[Proposition~1.74, p. 54]{IzmailovSolodov2014}. Hence, from Definition~\ref{def:DefSS} inequality~\eqref{eq:scA1} holds. Therefore, the  projected vector field $X$ satisfies the assumption \textbf{A1}.  In the following, we   present a concrete  example.

\begin{example}
Let  $Y: \mathbb{R}^{2} \to \mathbb{R}^{2}$ be a semismooth mapping defined by $Y(p) := Ap - |p| - b$  with matrix  $A=\mbox{diag}(4,3)$ and vector $b = (b_1, b_2) \in \mathbb{R}^{2}$,  where $\mbox{diag} (p_1, p_2)$  denotes a  $2\times 2$  diagonal matrix with $(i,i)$-th entry equal to $p_i$, $i = 1,2$. Take ${\bar p}=(0,1)\in \mathbb{S}^{2}$ and note that $Y(\bar{p}) = 0$ for $b = (0,2)$. Some calculus show that the Clarke generalized covariant derivative of $Y$ at $\bar{p}$ is given by  $ \partial Y({\bar p}) =\{\mbox{diag} (d, 2):~   d \in [3, 5] \}$.  Define  $X(p) := (I - pp^{T})Y(p)$  the  vector field on $\mathbb{S}^{2}$. Therefore,  using \eqref{eq:subX}, we conclude that $\partial X(\bar{p}) =\{ V_{\bar{p}}:= \mbox{diag}(d-2+b_2, -2+b_2):~ d \in [3, 5] \}$.  {\it Note that all $V_{\bar{p}}\in \partial X(\bar{p})$ are non-singular as a linear mapping $V_{\bar{p}}: T_{\bar{p}}\mathbb{S}^2 \to T_{\bar{p}}\mathbb{S}^2$, where the tangent hyperplane at $\bar{p}$ is given by  $T_{\bar{p}}\mathbb{S}^2:=\left\{v:=(v_1, 0)\in  \mathbb{R}^2 :~ v_1\in  \mathbb{R} \right\}$}. Hence, from Definition~\ref{def:reg}, we obtain that $X$ is regular at $\bar{p} = (0,1)$. Let $\lambda_{\bar{p}} \geq \max\{\|V_{\bar{p}}^{-1}\|: ~ V_{\bar{p}}\in \partial X(\bar{p})\}$. As $X$ is a locally Lipschitz continuous vector field,  using Lemma~\ref{le:NonSing} for every $\epsilon>0$ satisfying  $\epsilon \lambda_{\bar{p}} <1$, there exists $0 < \delta < \pi$  such that $X$ is regular on $ B_{\delta}(\bar{p})$ and for all $p\in   B_{\delta}(\bar{p})$ and ${V}_p \in \partial X(p)$ the following hold $\|{{V}_p}^{-1}\| \leq  \lambda_{\bar{p}}/(1 -\epsilon \lambda_{\bar{p}})$. Because $X$ is a semismooth vector field, we conclude that the assumption \textbf{A1} holds.
\end{example}

It is worth pointing out that in the literature there exist other examples of  semismooth vector field, see, for example, \cite{Hosseini2017}.

\section{Conclusions}\label{sec:fr}
In this paper, we studied the concept and some  properties of the  locally Lipschitz continuous vector fields. It is worth mentioning that  the Rademacher theorem is an essential tool  to  ensure   the existence of  Clarke generalized covariant derivative. Additionally, a version of Newton method for finding a singularity these vector fields was proposed. Under regularity and semismoothness the well-definedness  and local convergence of the method were established.  Furthermore,  a Kantorovich-type theorem was presented.  We expect that the results of this paper can aid in the extensions of new results and methods  of nonsmooth analysis to the Riemannian context, for example, the Mean Value Theorem as well as inexact and globalized versions of Newton method.

\section*{Acknowledgements}
 The work  was supported  by CAPES, FAPEG and CNPq Grants 408151/2016-1, 302473/2017-3.


\end{document}